\documentclass[12pt]{amsart}

\usepackage{amsmath,amssymb,amscd,amsthm}
\usepackage[all,cmtip]{xy} 
\usepackage{cite}

\usepackage{geometry} %[a4paper,divide={2cm,*,2cm},nofoot]%
\usepackage[all]{xypic}
\usepackage{graphicx}
\usepackage{mathrsfs}
\newtheorem{thm}{Theorem}[section]
\newtheorem{cor}[thm]{Corollary}
\newtheorem{lem}[thm]{Lemma}
\newtheorem{prop}[thm]{Proposition}
\theoremstyle{definition}
\newtheorem{defn}[thm]{Definition}
\theoremstyle{remark}
\newtheorem{rem}[thm]{Remark}
\numberwithin{equation}{section}

\newcommand{\smp}[1]{\left(#1 \right)}
\newcommand{\midp}[1]{\left[#1 \right]}
\newcommand{\bigp}[1]{\left\{#1 \right\}}

\newcommand{\R}{\mathbb{R}}

\newcommand{\N}{\mathbb{N}}

\DeclareMathOperator{\Sym}{Sym}

\DeclareMathOperator{\tr}{tr}

\providecommand{\abs}[1]{\left| #1\right|}
\providecommand{\norm}[1]{\lVert#1\rVert}

\DeclareMathOperator*{\osc}{osc}

\usepackage{esint}

\newcommand{\PB}{\mathcal{B}}
\newcommand{\Hp}[1]{\mathbf{H_{#1}}}
\newcommand{\Ud}{\mathcal{U}_{\delta}}
\newcommand{\wt}[1]{\widetilde{#1}}
\newcommand{\wh}[1]{\widehat{#1}}

\title{Small Perturbation Solutions for Parabolic Equations}
\author{Yu Wang}
\address{Yu Wang\\
Department of Mathematics\\
Columbia University, NY, U.S.} \email{yuwang@math.columbia.edu  }

\begin{document}

%Abstract..............................................................................................................
\begin{abstract}
Let $\varphi$ be a smooth solution of the parabolic equation
\[
F(D^2 u, Du, u,x, t) - u_t = 0.
\]
Assume that $F$ is smooth and uniformly elliptic only in a neighborhood of the points $(D^2 \varphi, D\varphi, \varphi,x,t)$, we show that a viscosity solution $u$ to the above equation is smooth in the interior if $\norm{u - \varphi}_{L^{\infty}}$ is sufficiently small.
\end{abstract}

\maketitle
%\tableofcontents

%% Content .............................................................................................................................................

\section{Introduction}
In this paper we present a general regularity result for small perturbation solutions of parabolic equations. It is a parabolic analogue to the result of \cite{Savin2}.

When dealing with a parabolic or elliptic equation, the classical approach to regularity is to differentiate the equation along a direction $e$. Then $u_e$ solves the linearized equation which is treated as a linear equation with measurable coefficients. When the equation is not uniformly elliptic, this approach requires \textit{a priori} bounds on $u, Du$ and $D^2u$. 

On the other hand, it is possible to obtain interior estimates from a nearby regular solution. This approach has been employed in several important works such as De Giorgi's analysis of minimal surfaces (see \cite{Giusti}) and Caffarelli's work on free boundary problems \cite{Caffarelli1}.

Recently Savin has applied this approach to study the flat level sets in Ginzburg-Landau phase transitions models \cite{Savin1}. He has also applied this approach to analyze small perturbation solutions of general elliptic equations \cite{Savin2}. The result in \cite{Savin2} has been employed by Armstrong, Silverstre and Smart to estimate the Hausdorff dimension of the singular set of a solution to a uniformly elliptic equation \cite{Armstrong1}. It has also been applied by de Silva and Savin to study the thin one-phase problem  \cite{DS}. 

It seems natural and interesting to establish an analogous result to \cite{Savin2} for parabolic equations.

\smallskip

Let $D$ represent the partial differentiation with respect to $x$-variables. Let $\Sym(n)$ be the space of $n \times n$ symmetric matrices equipped with standard spectral norm and $Q_1 = B_1 (0) \times (-1,0]$. Let
\[
F : \Sym(n) \times \R^n \times \R \times Q_1 \rightarrow \R \quad 
\]
be a function defined for $(M, p, z , x, t)$. Given a function $\varphi \in C^{2,\alpha} (Q_{1})$ and a number $\delta >0$, let
\[
\mathcal{U}_{\delta} (\varphi)= \{  (M + D^2 \varphi (x), p + D \varphi(x),z + \varphi (x),x, t) \;| \; \norm{M}, \abs{p}, \abs{z} < \delta , \; (x, t) \in Q_1\}
\] 
and
\[
F [\varphi] (x, t):= F (D^2\varphi (x, t), D \varphi (x,t), \varphi (x,t), x, t) .
\]

\smallskip

In this paper, we shall consider parabolic equations of the form
\begin{equation}
\label{Eq1}
F [u] (x, t) - u_t = 0
\end{equation}
under the following hypotheses regarding $F$ on $\Ud (\varphi)$:

$\Hp{\varphi}1)$ $F (\cdot, p, z, x, t)$ is elliptic, i.e., for every $(M,p,z,x,t) \in \Ud (\varphi)$,
\[
F (M + N, p , z, x,t) \geq F (M, p,z,x ,t),  \quad  \forall  N \geq 0.
\]

$\Hp{\varphi}2)$ $F (\cdot, p,z,x,t)$ is uniformly elliptic in $\Ud(\varphi)$, i.e.,  $ \exists \; \Lambda \geq 1\geq \lambda >0$ such that for all $(M, p,z,x,t) \in \mathcal{U}_{\delta} (\varphi)$,
\[
\Lambda \norm{N} \geq F (M + N, p,z,x,t) - F(M, p,z,x,t) \geq \lambda \norm{N}, \quad  \forall N \geq 0, \;  \norm{N} \leq \delta.
\]

$\Hp{\varphi}3)$ $\varphi$ is a solution to \eqref{Eq1}.

$\Hp{\varphi}4)$ $F\in C^1 (\Ud (\varphi))$ and 
\[
\norm{\nabla F}_{L^{\infty} (\Ud(\varphi))} \leq K, 
\]
where derivatives of $F$ are taken with respect to all variables $(M, p, z, x, t)$.
\medskip

$\Hp{\varphi}5)$ $\nabla_{M}F $ (derivatives of $F$ with respect to the matrix variables $M \in \Sym (n)$) have uniform continuity, i.e., there exists an increasing continuous function $\omega: [0, \infty) \rightarrow [0, \infty)$ such that $\omega (0) = 0$ and for every pair $(A, p,z,x, t),(B,p,z,x,t ) \in \Ud (\varphi)$,
\[
\norm{\nabla_{M} F (A,p,z,x, t) - \nabla_{M} F(B,p,z,x,t)} \leq \omega (\norm{A - B}).
\]

Note that we do not require any information about $F$ outside $\Ud (\varphi)$. A large class of operators satisfy the above conditions, for example, the real and complex Monge-Amp\`ere operators.

All solutions mentioned in this paper are understood in the viscosity sense (see \cite{Wang1}). From now on, we refer to positive constants that depend only on ($n, \lambda, \Lambda$) as universal constants and positive constants that depend only on $n$ as dimensional constants. We shall label the dependence explicitly if a constant depends on other parameters $(K, \delta, \omega$).  

Next we recall the standard conventions (see \cite{Lieberman}). A function $u$ defined on $Q_1$ is said to be $C^{k, \alpha}$ with $k$ being even if 
\[
\norm{u}_{C^{k,\alpha}} := \sum_{j=1}^{k/2-1} \sum_{i=1}^{k-1} \norm{D^{i} D^j_{t}u}_{L^{\infty}} + \norm{D^k u}_{C^{\alpha}} + \norm{D^{k/2}_t u}_{C^{\alpha/2}} < \infty.
\]

Our main result is the following theorem.
\begin{thm}
\label{Main}
Suppose that $F$ satisfies $\Hp{\varphi}1)- \Hp{\varphi}5)$ with $\varphi \equiv 0$. Then for each $\alpha \in (0,1)$,  there exist positive constants $(\mu_0 , C_0)$ only depending on $(n, \lambda, \Lambda, K, \delta, \omega, \alpha)$ such that the following statement holds:

If $u$ is a solution to \eqref{Eq1} in $Q_1$ and satisfies that
\[
\norm{u}_{L^{\infty} (Q_1)} \leq \mu_0,
\]
then $u \in C^{2,\alpha} (Q_{1/2})$ and
\[
\norm{ u}_{C^{2} (Q_{1/2})} \leq \delta,  \quad   \norm{u}_{C^{2,\alpha} (Q_{1/2})} \leq C_0.
\]
\end{thm}

The above theorem is parallel to the main result in \cite{Savin2} with a slight refinement on the structure condition of $F$ ($D^2 F$ is required to be bounded in \cite{Savin2}). Such a refinement in the elliptic case has been pointed out in \cite{Armstrong1}. As an immediate consequence of Thm.\ref{Main}, we have

\begin{cor}
\label{cmain}
Let $\alpha \in  (0,1)$ and $\varphi \in C^{3, \alpha} (Q_1)$. Suppose that $F$ satisfies $\Hp{\varphi}1)- \Hp{\varphi}5)$. Then there exist positive constants $(\mu_1, C_1)$ only depending on $(n,\lambda, \Lambda, K, \delta, \omega, \alpha, \norm{\varphi}_{C^{3,\alpha} (Q_1)} )$ such that the following statement holds: 

If $u$ is a solution to Eq.\eqref{Eq1} in $Q_1$ and satisfies that
\[
\norm{u - \varphi}_{L^{\infty} (Q_1)} \leq \mu_1,
\]
then $u \in C^{2,\alpha} (Q_{1/2})$ and
\[
\norm{u - \varphi}_{C^{2} (Q_{1/2})} \leq \delta, \quad \norm{u - \varphi}_{C^{2,\alpha} (Q_{1/2})} \leq C_1.
\]
\end{cor}

The dependency of $(\mu_{1}, C_{1})$ on $\norm{\varphi}_{C^3, \alpha}$ can be effectively reduced according to the specific structure of $F$. For example, if $F = F(M,x )$, only $\norm{D^3 \varphi}_{L^{\infty}}$ instead of $\norm{\varphi}_{C^{3,\alpha}}$ will enter into the constant dependency in Cor.\ref{cmain}.

Along the proof of Thm.\ref{Main} we also produce some other important results. In particular, we establish the oscillation decay property for the solutions to Eq.\eqref{Eq1} in the case that $F : \Sym (n) \rightarrow \R$ only satisfies $\Hp{0}1)-\Hp{0}3)$ (see {Prop.\ref{OD}). 

\medskip

We follow closely the method in \cite{Savin2}. Recall that the main ingredient of the proof there is to establish certain homogeneity of contact sets with respect to concentrated balls (Lem.2.2 in \cite{Savin2}). The key step in this paper is to establish a parabolic analogue of this homogeneity. From this we deduce the oscillation decay property of solutions (Prop.\ref{OD}), based on which we may perform a blow-up argument to obtain $C^{2,\alpha}$-regularity. 

Our study here also shares many similar ideas to \cite{Wang1}, \cite{Wang2} and \cite{Wang3} which dates back to \cite{Caffarelli2} and \cite{Krylov}. However, our local analysis differs from \cite{Wang1} in techniques. Rather than working with cubes in $\R^n$ and parabolic cylinders, we perform our analysis over a class of special domains - parabolic balls (Defn.\ref{PB}). A parabolic ball can be viewed as the union of parabolic cylinders with all scales. Indeed, even not used explicitly, it has been pointed out in \cite{Wang1} that one should view parabolic cylinders in a scaled fashion (P. 30 in \cite{Wang1}). This point of view should correspond to the notion of parabolic balls used here. Although we cannot find explicit reference regarding parabolic balls in literature, we believe that similar notions have been considered by many authors.

In order to present the main idea in a transparent fashion, we shall first discuss the case that $F$ only depends on $M \in \Sym (n)$, i.e., $F$ is of the form
\begin{equation}
\label{Simple}
F : \Sym (n) \rightarrow \R.
\end{equation}
Then in a seperate section (\S 6), we will explain how the proof  can be adapted to establish Thm.\ref{Main} in the general case.

\medskip

The paper is organized as follows: In \S2 we study the basic properties of parabolic balls and contact sets. In \S3 the homogeneity of contact sets with respect to parabolic balls is proved. In \S4 we establish the oscillation decay property of solutions. \S5 is devoted to the proof of Thm.\ref{Main} under the assumption that $F$ is of the form \eqref{Simple}. In \S6 we explain how to modify the proof in \S5 to establish Thm.\ref{Main} for general setting and give the proof of Cor.\ref{cmain}.

\section{Parabolic Balls and Contact Sets}
In this section we introduce the notion of parabolic balls and parabolic contact sets and list their basic properties. All distances and measures are taken to be the standard $(n+1)$-dimensional Euclidean distances and Lebesgue measures.

\begin{defn}
\label{PB}
We define parabolic balls of opening $\theta >0 $ to be domains of the following forms:
\begin{equation}
\begin{split}
& \PB_{T}^{\theta} (x_0, t_0) := \bigp{ (x, t) \; | \;  \theta \abs{x - x_0}^2 \leq t-t_0 \leq T} ,  \\
&  \PB_{T}^{-\theta} (x_0, t_0) := \bigp{ (x, t) \; | \; \theta \abs{x - x_0}^2\leq t_0-t \leq T}.
\end{split}
\end{equation}
\end{defn}

By direct calculation, 
\begin{equation}
\label{VPB}
\abs{\PB_{T}^{\theta} (x,t)} = \abs{\PB_{T}^{-\theta} (x,t)} = \tfrac{2\omega_n}{n+2} T^{1 + n/2} \theta^{-n/2}
\end{equation}
where $\omega_{n}$ is the volume of the unit ball in $\R^n$.

\smallskip

Now we study some intersection properties of parabolic balls. Although these properties are established via elementary arguments, they seems not standard in the literature. Thus we shall provide more details in the proof.

Recall the standard notation of parabolic cylinders $
Q_{r} (x, t) := B_{r} (x) \times (t- r_0^2, t ] $ and $Q_{r} = Q_{r} (0, 0)$.

\begin{lem}
\label{intersection1}
Let $\theta \geq \frac{3}{4}$. For every $(x_1, t_1) \in \PB_{T_0}^{-\theta} (x_0, t_0)$ and $0<T \leq t_0 -t_1$, there exists a cylinder $Q_{r} (x_2, t_2)$ such that:

i) $Q_{r} (x_2, t_2) \subset \PB_{T}^{\theta} (x_1, t_1 ) \cap  \PB_{T_0}^{-\theta} (x_0, t_0) \cap \{(x, t) : t_1 + T/4\leq t \leq t_1 + T/2 \}$;

ii) $\abs{Q_r (x_2, t_2)} / \abs{\PB_{T}^{\theta} (x_1, t_1 )} \geq   \eta_{0}$, where $\eta_0$ only depends on the dimension.

iii) If $(y_0,s_0) \in Q_{r/4} (x_2, t_2)$, then $
\PB_{s_0-s}^{1/2} (y,s) \subset Q_{r} (x_2, t_2)$ for every $(y,s) \in \PB^{-1/2}_{r^2/16} (y_0, s_0 - \tfrac{r^2}{16})$.
\end{lem}

\begin{proof}
For simplicity of notations, set $\PB_{0}^- := \PB_{T_0}^{-\theta} (x_0, t_0)$ and $\PB_{1}:= \PB_{T}^{\theta} (x_1, t_1)$. By scaling, we may assume $t_0 - t_1 =1$. 

\begin{center}
\includegraphics[width=9cm, height=5cm]{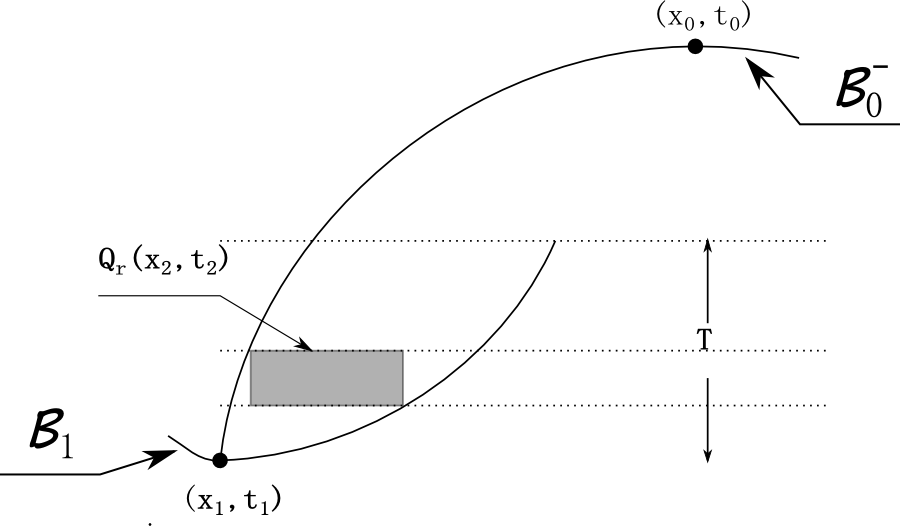}

Figure 2.1
\end{center}

Set $r_0 = \sqrt{\tfrac{(1- T/2)}{\theta}}, r_1= \sqrt{\tfrac{T}{4\theta}}$, $r_2=(r_0 +r_1-\abs{x_1-x_0})/2$ and
\[
x_2 := (1-\delta )x_0 + \delta  x_1, \quad  \delta = (r_0 -r_2)/\abs{x_1 - x_0}. 
\]
We shall show that
\[
Q_{r} (x_2, t_2), \quad  t_2 =  t_1 + T/2, r = \min\{r_2, \sqrt{T}/2\}
\]
satisfies all desired properties. 

\smallskip

To show i), let $\mathrm{Proj} : \R^{n+1} \rightarrow \R^n$ be the projection given by $(x, t) \mapsto x$ and
\[
E:= \mathrm{Proj} ( \PB_0^- \cap \{(y,s) : s = t_1 + T/2\}) \cap \mathrm{Proj} ( \PB_1 \cap \{(y,s) : s = t_1 + T/4\}).
\]
It is easy to see that
\[
B_{r} (x_3 ) \subset E = B_{ r_1} (x_1) \cap B_{r_0} (x_0) \subset \R^n, 
\]

Now, by the definition of parabolic balls, 
\[
\mathrm{Proj} ( \PB_0^- \cap \{(y,s) : s = t_1 + T/2\}) \subset \mathrm{Proj} (\PB_0^- \cap \{(y,s) : s =t\} )  , \quad \forall t < t_1+T/2
\]
and
\[
\mathrm{Proj} ( \PB_1 \cap \{(y,s) : s = t_1 + T/4\} )  \subset \mathrm{Proj} (\PB_1 \cap \{(y,s) : s =t\} ) , \quad \forall t > t_1 +T/4.
\]

Thus
\[
B_{r}(x_3) \subset \mathrm{Proj} (\PB_0^- \cap \PB_1 \cap \{(y,s) : s =t\} ) , \quad \forall t \in [t_1 +T/4, t_1 +T/2].
\]
Along with the definition of $r$, we have,
\[
Q_{r} (x_2, t_2) \subset \PB_{T}^{\theta} (x_1, t_1 ) \cap  \PB_{T_0}^{-\theta} (x_0, t_0) \cap \{(x, t) : t_1 + T/4\leq t \leq t_1 + T/2 \}.
\]
This proves i).

To show ii), by direct calculation, we have
\begin{equation}
\label{Lem22r}
 r  \geq  \tfrac{\sqrt{2}-1}{4} \sqrt{\tfrac{T}{\theta}}   , \quad \forall T\leq1= t_0 - t_1, .
\end{equation}
By Eq.\eqref{VPB}, Eq.\eqref{Lem22r} and assumption that $\theta \geq 3/4$ , we have
\[
\abs{Q_{r} (x_2, t_2)} \geq (\tfrac{\sqrt{2} -1}{4})^{-(n+1)} T^{1+n/2} \theta^{1-n/2} \geq  \eta_0 \abs{\PB_{T}^{\theta} (x_1, t_1)}.
\]
This proves ii). 

To show iii), fix a $(\xi, \tau) \in \PB^{1/2}_{r^2/16} (y ,s )$, we have $s_{0}-r^2/8 \leq \tau \leq s_0 \leq t_2 $ and
\[
\abs{\xi - x_2} \leq \abs{\xi - y} + \abs{y - y_0} + \abs{y_0 - x_2}  \leq \tfrac{\sqrt{2}r }{4} + \tfrac{\sqrt{2} r}{4} + \tfrac{r}{4} \leq r.
\] 
This proves iii).
\end{proof}

\smallskip

\begin{lem}
\label{intersection2}
Let $(x_0, t_0) $ and $(x_1, t_1) $ be two points in $\overline{Q}_{11/24} $. For $i=0,1$, set 
\[
\theta_i := 
\min\{\theta \; |\;  \PB_{1+ t_i}^{-\theta} (x_i, t_i) \subset \overline{Q}_1\},
\]
and $
\PB_{i}^{-} := \PB^{\theta_i}_{1+t_i} (x_i, t_i).$
Then there exists a dimensional constant $\eta_1$ such that
\[
\tfrac{\abs{\mathcal{E}} }{\abs{\PB_0^-} } , \tfrac{\abs{ \mathcal{E}  }}{\abs{\PB^-_1}} \geq \eta_1,  \text{ where } \mathcal{E} = \PB_0^- \cap \PB_1^-.
\]
\end{lem}

\begin{center}
\includegraphics[width=6cm, height=3.5cm]{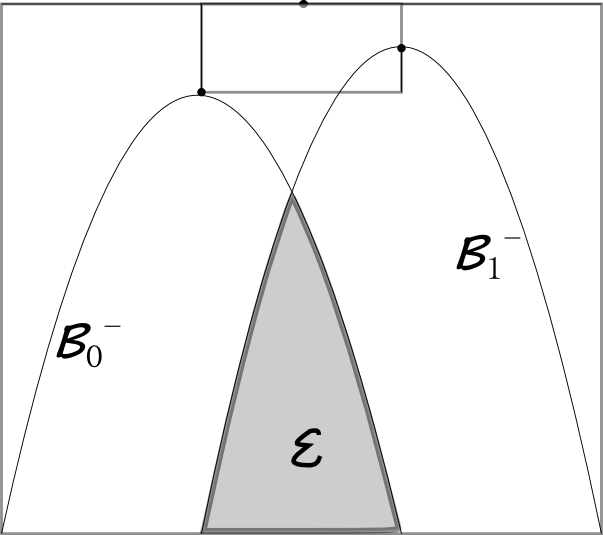}

Figure 2.2
\end{center}

\begin{proof}
Consider $\mathrm{Proj}  (\PB_{0}^- \cap \{t = -1\}) $ and $\mathrm{Proj} (\PB_{1}^- \cap \{t=-1\} )$. Set $r_i = 1- \abs{x_i}, i =0,1$ and
\[
x_2 := \frac{r_1 -r}{\abs{x_0 - x_1}} x_0 + \frac{r_0-r}{\abs{x_0 - x_1}} x_1, \quad  r=\frac{(  r_0 + r_1 - \abs{x_0-x_1} ) }{2}.
\]
Then
\[
B_{r} (x_2 ) \subset B_{r_0} (x_0) \cap B_{r_1} (x_1 ) \subset \R^n.
\]
For each $y \in B_{r/2} (x_2)$, let $s_i (y) > 0$ be the unique number such that $(y, -1 + s_i)   \in \partial \PB_i^-$ and let
\[
s  :=  \inf \{ \min \{s_0 (y), s_1 (y)\} \; |\;  y\in B_{r/2} (y) \}
\]
Clearly
\[
B_{r/2} (x_2 ) \times [-1, -1 +s] \subset \mathcal{E}.
\]
Since $(x_0 ,t_0), ( x_1,t_1) \in Q_{11/24}$, 
\[
r = 2 - (\abs{x_1} + \abs{x_0} +\abs{x_1 - x_0}) \geq 1/4.
\]
It then suffices to find a lower bound of $s$ that is independent of $(x_0,t_0)$ and $(x_1, t_1)$.

First, we estimate $s_0 (y)$ from below. By the definition of $\theta_0$ and $s_0(y)$, 
\[\begin{split}
s_0 (y) & = 1+t - \theta_0 \abs{y - x_0}^2 = (1 + t) \smp{1 - \tfrac{\abs{y -x_0}^2}{r_1^2}}.
\end{split}\]
By the choice of $x_2$ and the fact that $y \in B_{r/2} (x_2)$, we obtain
\[
\abs{y - x_0} \leq r_1 - r/2.
\]
Along with the fact that $\abs{t} \geq (11/24)^2$, we obtain
\[
s_0 (y)  \geq \frac{1}{2} \smp{ \frac{r}{r_1} + \frac{r^2}{ 4 r_1^2}} \geq \frac{1}{8}, \quad \forall y \in B_{r/2} (x_2 ).
\]
With the same procedures, one obtains $s_1 (y) \geq 1/8,  \forall y \in B_{r/2} (x_2 )$. This completes the proof.
\end{proof}

Next we establish a Vitali-type of covering lemma for parabolic balls. Consider the following construction.

Given $(x_1,t_1) \in \PB_0^- = \PB_{T_0}^{-\theta} (x_0, t_0) $ and a parabolic ball $\PB_1 := \PB_{T_1}^{\theta} (x_1, t_1)$. Let 
\begin{equation}
\label{hatPB}
\wh{\PB}_1 := \PB_{4T_1}^{\hat{\theta}} (x_1, t_1- 3T_1),  \quad   \hat{\theta} = \frac{\theta}{(\sqrt{2}+1)^2}.
\end{equation}
By Eq.\eqref{VPB},
\begin{equation}
\label{chat}
\eta_{2} := \abs{\PB_{T}^{\theta} (x, t)} /\abs{\wh{\PB}_{T}^{\theta} (x, t)}  = 4^{-(1+n/2)} (\sqrt{2} +1)^{-n}.
\end{equation}

\begin{lem}
\label{VC}
Let $D$ be a bounded subset of $\R^{n+1}$ and $T_{x,t} = T (x,t)$ be a positive function on $D$. Let 
\[
\mathfrak{F}:=\bigp{ \PB_{T_{x,t}}^{\theta} (x,t) \; |\;  (x,t) \in D }.
\]
Suppose that 
\[
 T_0 :=  \sup \{ T(x,t) \; | \;  (x, t) \in D \} <\infty.
\]
Then there exists a countable disjoint sub-collection $\{ \PB_i =\PB_{T_{x_i, t_i}}^{
\theta} (x_i, t_i) : i \in \N\}$ such that
\[
D \subset \bigcup_i \wh{\PB_i}  .
\]
\end{lem}

\begin{proof}
We shall mimic the standard proof of the Vitali's covering lemma (see \cite{EG} Page 27). Let
\[
\mathfrak{F}_{k}:=\{ B_{T}^{\theta} (x,t)  \in \mathfrak{F} \; |   \;   2^{-(k+1)} T_0 <T \leq  2^{-k} T_0  \}, \quad k \in \N.
\]
We define $\mathfrak{G}_{k} \subset \mathfrak{F}_k$ as follows:

a) Let $\mathfrak{G}_0$ be a maximal disjoint sub-collection of $\mathfrak{F}_0$.

b) Assuming $\mathfrak{G}_0, \mathfrak{G}_1, .., \mathfrak{G}_{k-1}$ have been selected, we choose $\mathfrak{G}_k$ to be a maximal disjoint sub-collection of 
\[
\bigp{\PB \in \mathfrak{F}_{k} \; | \; \PB \cap \PB' = \emptyset , \forall \PB' \in \bigcup_{j=0}^{k-1} \mathfrak{G}_j }.
\]

Finally, define $\mathfrak{G}:= \cup_{k=0}^{\infty} \mathfrak{G}_k$. Clearly $\mathfrak{G}$ is a collection of disjoint parabolic balls and $\mathfrak{G} \subset \mathfrak{F}$. 

First, we show that $\mathfrak{G}$ is countable. Since $(x,t) \in D$ and $T_0 < \infty$, $\cup_{\mathfrak{F}} \PB$ is a set with finite volume. Meanwhile, for each fixed $k$, 
\[
\abs{\PB} > 2^{-k(1+n)/2} \tfrac{2\omega_n}{n+2} T_0^{1 +n/2} \theta^{-n/2} , \quad \forall \PB \in \mathfrak{G}_k.
\] 
Therefore, $\mathfrak{G}_k$  must be a finite set, because it consists of disjoint parabolic balls. It follows that $\mathfrak{G}$ is countable. 

Next, we show $D \subset \cup_{\mathfrak{G}} \wh{\PB}$. Given any point $(x,t) \in D $, consider $\PB (x,t):= \PB_{T_{x,t}}^{\theta} (x,t) \in \mathfrak{F}$. Let $k$ be the unique integer such that $\PB (x,t) \in \mathfrak{F}_k$. If $\PB(x,t) \in \mathfrak{G}_k$, then clearly $(x,t) \in \cup_{\mathfrak{G}} \wh{\PB}$. If $\PB (x,t) \not \in \mathfrak{G}_k$, then by the maximality of $\mathfrak{G}_k$, there exists a $\PB(y,s) := \PB_{T_{y,s}}^{\theta} (y,s)\in \cup_{j=0}^{k} \mathfrak{G}_j$ such that $\PB(x,t) \cap \PB (y,s) \neq \emptyset $. 

\begin{center}
\includegraphics[width=7cm, height=2.5cm]{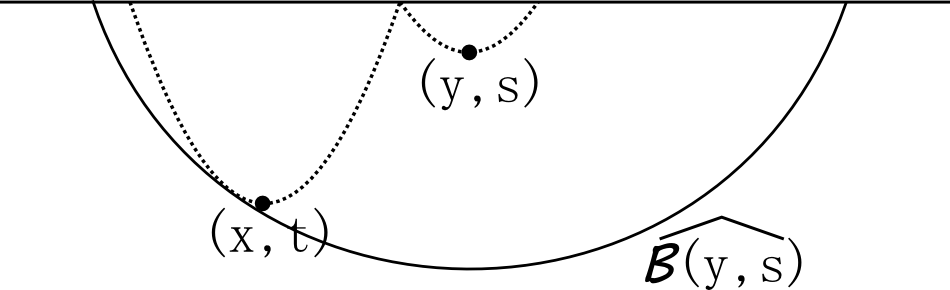}

Figure 2.3
\end{center}

Claim: $(x,t)  \in \widehat{\PB(y,s)}$. For simplicity of the notations, let $T = T_{y,s}$. By the construction, $T_{x,t} \leq 2 T$. Let $(y',s') \in \PB(x,t) \cap \PB (y,s)$. Then,
\[
\abs{x - y} \leq \abs{x - y'} + \abs{y'-y} \leq (\sqrt{2} + 1) \sqrt{T/\theta}
\]
and
\[
t -(s-3T) \geq s' - T_{x,t} -(s -3T) \geq T.
\]
Therefore,
\[
\hat{\theta} \abs{x - y}^2 \leq T \leq (t-(s-3T)),
\]
which implies that $(x,t) \in \wh{\PB(y,s) }$. This completes the proof.
\end{proof}

\medskip

Now we move to the discussion of contact sets. We begin with some terminologies. A function $P(x,t)$ on $\R^n \times \R$ is called a concave parabola of opening $a >0$ if it is of the form
\[
P_{(y,s; a)} (x, t) = -\frac{a}{2} \abs{x - y}^2  + a (t-s),
\]
for some point $(y,s) \in \R^{n} \times \R$; It is called a convex parabola if it is of the form
\[
P_{(y,s; a)} (x, t) = \frac{a}{2} \abs{x - y}^2  - a (t-s).
\]

Let $\varphi$ be a smooth function and $u \in C(\overline{B_1} \times \R)$. We say that $\varphi$ contacts $u$ at $(x,t)$ from below if 
\[
\varphi (\xi, \tau) < u (\xi , \tau) , \; \forall \xi \in B_1 , \tau < t
 \quad \text{ and } \quad
\varphi  (x, t) = u (x, t).
\]
Similarly, we say that $\varphi$ contacts $u$ at $(x,t)$ from above if
\[
\varphi (\xi, \tau) > u (\xi , \tau) , \; \forall \xi \in B_1 , \tau < t
 \quad \text{ and } \quad
\varphi  (x, t) = u (x, t).
\]
Note that we always compare the function $\varphi$ with $u$ in time $\tau \leq t$.

\medskip

\begin{defn}
Let $u \in C(\overline{B}_1 \times \R)$ be a bounded function and $E$ be a compact subset of $\R^{n}\times \R$. Given $a>0$, the contact set $A_{a}(E,u)$ is defined as follows: 
\[\begin{split}
A_{a} (E,u) := \{& (x, t)  \in \overline{B}_1 \times \R | \;  \exists (y,s) \in E \text{ s.t. } \\
& P_{(y,s ; a)} \text{ contacts } (u -\min_{Q_1}u ) \text{ from below at } (x,t)  \}.
\end{split}\] 
We shall write $A_{a} (E)$ if the indication of the function $u$ is clear.
\end{defn}

The following lemma is a summary of the basic properties of $A_a (E)$. The proof is straight forward, hence it is omitted.

\begin{lem}
\label{contact set}
Let $u \in C(\overline{B}_1 \times \R), a, b \in \R^{+}$ and $E, F$ be compact subsets of $\R^{n}\times \R$. Then

i) $A_a (E)$ is closed.

ii) If $E \subset F$, then $A_a (E) \subset A_a  (F)$.

iii) If $a \leq b$, then $\smp{ A_{a} (\overline{Q}_1)  \cap Q_1} \subset \smp{ A_{b} (\overline{Q}_1)   \cap Q_1} $.

\end{lem}

We end up this section with a version of ABP-estimate regarding $A_a (E)$. This estimate follows an idea similar to \cite{Tso}. As in the usual discussion of parabolic equations, all continuity and differentiability with respect to the $t$-variable are understood as the left-continuity (differentiability) if necessary. We first recall the following definition and theorem.
\begin{defn}
A function $u \in C(\overline{Q}_1)$ is said to be locally uniformly semi-concave in $Q_1$ if for each compact subset $Z$, there exists a constant $b>0$ such that for each $(x,t) \in Z$, there exists a convex parabola $P_{b}$ of opening $b$ contacting $u$ from above.
\end{defn}

\begin{thm}
\label{Alex}
If $u$ is locally uniformly semi-concave in $Q_1$, then there exists a measure-zero set $\mathcal{N}$ such that the following statement holds: 

For each $(x, t) \in Q_1 \setminus \mathcal{N}$, there exists a quadratic polynomial
\[
p_{x,t}(\xi,\tau) = a + b\cdot (\xi - x)  + \beta (\tau - t)+ \frac{1}{2} (\xi - x)^t M (\xi - x), \quad a, \beta\in \R,b \in \R^n, M \in \Sym(n), 
\]
such that
\[\begin{split}
u (\xi, \tau) =p_{x,t} (\xi, \tau)+o (\abs{\xi -x}^2 +\abs{\tau -t}), \quad \abs{\xi -x} ,\abs{\tau - t} \rightarrow 0.
\end{split}\]
\end{thm}

\begin{proof}
Since $u$ is locally uniformly semi-concave, for each compact subset $Z$, there exists $b$ such that
\[
\tilde{u} (x, t)  := u (x, t) - (b\abs{x}^2/2 - bt) ,\quad (x,t) \in Z
\]
is concave in $x$ and monotone in $t$. Thus, by applying the parabolic version of Alexandrov's differentiability theorem (see Appendix 2 of \cite{Krylov}, Theorem 1 on Page 444) to $\tilde{u}$, the desired conclusion follows. 
\end{proof}

The above theorem allows one to define $Du, D^2 u, u_t$ at $(x,t)  \in Q_1 \setminus \mathcal{N}$. In the rest of this paper, we shall understand derivatives of $u$ away from $\mathcal{N}$ in the above sense.

Here comes our version of ABP-estimate.
\begin{lem}
\label{ABP1}
Let $u \in C(\overline{B}_1 \times \R)$ be locally uniformly semi-concave in $Q_1$. Suppose that
\[
A_{a} (E) \subset  Q_1.
\]
Then
\[
\abs{E} \leq \int_{A_{a} (E) \setminus \mathcal{N}'} \det \smp{I +a^{-1}D^2u (x,t)} \smp{1 -a^{-1} u_{t} (x,t)} \; dxdt.
\]
where $\mathcal{N}'$ is a measure-zero set.
\end{lem}

\begin{proof}
Without loss of generality, we may assume that $\min_{Q_1} u  = 0 $. Since $A_{a} (E) $ is a compact subset of $Q_1$, the uniform semi-concavity allows us to find a contact parabola of opening $b$ from above for each $(x,t) \in A_{a } (E)$. The constant $b$ is independent of $(x,t)$. On the other hand, by the definition of contact set, for each $(x,t)\in A_{a}(E)$, there exists $(y,s) \in E$ such that $P_{(y,s;a)}$ contacts $u$ at $(x,t)$. Thus on each $(x, t) \in A_a (E)$, $u$ is contacted from above and below by parabolas; therefore, $u$ is differentiable at $(x,t)$.

By the contact condition and the assumption that $A (E) $ lies in the interior of $Q$, we have
\begin{equation}
\label{Tmap}
\begin{cases}
y = x + a^{-1} \nabla u (x, t)  \\
s =t - a^{-1}  u (x, t) - \frac{1}{2} \abs{x -y}^2  ,
\end{cases} \quad
 \end{equation}
and
\begin{equation}
\label{ContactR}
u_{t} \leq a,  \quad  \;  D^2 u \geq - a I_n .
\end{equation}
where $I_n$ is the $n\times n$ identity matrix.

Let $T$ be the mapping that maps $(x,t)$ to $(y,s)$ according to Eq.\eqref{Tmap}. It is easy to check that $T$ is a Lipschitz map and
\[
\abs{T(x,t) - T(x',t')} \leq C (\abs{x -x'} + \abs{t -t'}),  \quad \text{ where } C \text{ only depends on } a, b.
\]
Meanwhile, by the boundedness and continuity of $u$, for each $(y,s) \in E$, there exists $(x,t) \in A_{a} (E)$ such that $P_{(y,s ; a)} $ contacts $u$ at $(x,t)$. Hence $T$ is a surjective Lipschitz map from $A_a (E)$ to $E$. 

Apply the Area formula, we obtain  
\[
\abs{E} = \abs{T (A_a (E))} = \int_{A_a (E ) \setminus S}\abs{  \det{D_{x,t } T} }  \; dxdt
\]
for any measure-zero set $S$.

To prove the lemma, we are left to compute the Jacobian of $T$. Let 
\[
\mathcal{N'} =  \mathcal{N} \cup \{ (x, t) \in A_a (E)| \; T \text{ is not differentiable at } (x, t)\},
\]
where $\mathcal{N}$ is given in Thm.\ref{Alex}. Since $T$ is Lipschitz, $\mathcal{N'}$ has zero measure. Moreover, by \eqref{Tmap},  for each $(x, t) \in A_a (E) \setminus \mathcal{N}' $, 
\[
D_{x, t} T (x,t) = \begin{pmatrix}
D_x y & \partial_t y \\
D_x s & \partial_t s
\end{pmatrix} = \begin{pmatrix}
I + a^{-1}D^2 u(x, t) & \partial_{t} y\\
0 & 1- a^{-1} u_{t} (x,t)
\end{pmatrix} .
\]
By the contact relation, the diagonal entries are all nonnegative. The desired estimate then follows immediately.
\end{proof}

\section{Homogeneity of Contact Sets}
In this section we establish the main ingredient (Prop.\ref{HoC}) in proving Thm.\ref{Main}.

\begin{prop}[Homogeneity of Contact sets]
\label{HoC}
Let 
\begin{equation}
\label{Constant0}
c_0 = \frac{\lambda^2}{\Lambda^2 (n+5)} \exp\{ - 10^3 \Lambda n /\lambda\}
\end{equation}
and $\PB^{-\theta}_{T_0} (x_0, t_0) \subset Q_1$. Let $F: \Sym (n) \rightarrow \R$ satisfy $\Hp{0}1) - \Hp{0}3)$ and $u \in C (\overline{B}_1 \times \R) $ be a bounded and locally uniformly semi-concave function. 

Suppose that $\tfrac{3}{4} \leq \theta \leq 4$ and
\[
F (D^2u ) - u_t \leq 0 , \quad \text{ in } Q_1.
\]
Then there exists a universal constant $c_1$ such that the following statement holds:

For every  $a \in (0, c_0 \delta) $ and every $\PB_{T_1}^{\theta} (x_1, t_1)$ with $(x_1,t_1) \in \PB^{-\theta}_{T_0} (x_0, t_0)$, if
\[
(\PB_{T_1}^{\theta} (x_1, t_1) \cap \{(x, t) |\; t = t_1 + T_1\})\cap (A_a (\overline{Q}_1) \cap \PB_{T_0}^{-\theta} (x_0,t_0) ) \neq \emptyset   ,
\]
then 
\[
 \abs{ A_{c_1^{-1} a}  (\overline{Q}_1 ) \cap \PB^{-\theta}_{T_0} (x_0, t_0)   \cap  \PB^{\theta}_{T_1} (x_1, t_1)  }  \geq c_1  \abs{\PB^{\theta}_{T_1} (x_1, t_1)} .   
\]
\end{prop}

We shall need several lemmas to prove Prop.\ref{HoC}. All lemmas in this section are stated under the assumptions of Prop.\ref{HoC}.

\begin{lem}
\label{LB}
If $M \geq - a I $ and $F (M ) \leq a$ ,then $\norm{M} \leq \delta$.
\end{lem}

\begin{proof}
Argue by contradiction. Otherwise, there exists a direction $e$ such that 
\[
M >  (a + \delta) e\otimes e - a I.
\]
By $\Hp{0}1) $, $\Hp{0}2)$ and $\Hp{0}3)$,
\begin{equation}
\label{UseEq1}
\begin{split}
a \geq & F (M, p,z,x,t) \geq F\smp{ (a+ \delta) e\otimes e - aI} \\
 & \geq \lambda \delta - \Lambda (n-1)  a.
\end{split}
\end{equation}
This leads to a contradiction by the choice of $c_0$.
\end{proof}

\begin{lem}
\label{ABP2}
Let $E$ be a compact subset of $\R^{n} \times \R$.  if  $
A_{a} (E) \subset Q_1$,
then
\[
\abs{A_a (E)} \geq c_2 \abs{E},  \quad \text{ where } c_2 =  \smp{1 + \Lambda n/ \lambda}^{n+1}.
\]
\end{lem}

\begin{proof}
By Lem.\ref{ABP1}, it suffices to control 
\[
\det \smp{I +a^{-1}D^2u (x,t)} \smp{1 -a^{-1} u_{t} (x,t)}
\]
from above on $A_{a} (E) \setminus \mathcal{N}'$. Let 
\[
p_{x,t} (\xi, \tau) =u(x,t)+ b\cdot (\xi - x)  + \beta (\tau  - t)+ \frac{1}{2} (\xi - x)^t M (\xi -x)
\]
be the quadratic polynomial given by Thm.\ref{Alex}. By the contact relation,
\[
M \geq - a I, \quad  \beta \leq a.
\]

For every $\epsilon>0$,
\[
p_{x,t} (\xi, \tau) - \epsilon \smp{\abs{\xi -x}^2 - (\tau - t)}, \quad \tau < t
\]
contacts $u$ from below at some point in $B_{r} (x) \times (t-r^2 , t]$ with some small $r$ depending on $\epsilon$. By the definition of viscosity super-solution, 
\[
F(M - \epsilon I ) \leq \beta \leq a.
\]
Let $C a$ be the largest eigenvalue of $M$ and $e$ be the corresponding eigenvector. By Lem.\ref{LB}, $(C a-\epsilon )\leq \delta$. Then by $\Hp{0}1), \Hp{0}2)$ and $\Hp{0}3)$,
\begin{equation}
\label{UseEq2}
\begin{split}
 a \geq F (M - \epsilon I ) &\geq F\smp{Ca e\otimes e  - (a+\epsilon)I } \\
 & \geq \lambda [ (C-1)a - \epsilon ]- \Lambda (n-1) (a + \epsilon).
\end{split}
\end{equation}
Let $\epsilon$ tend to $0$, we obtain
\[
D^2 u (x, t) \leq \Lambda n/\lambda a.
\]
On the other hand,
\[
 \beta \geq F (M - \epsilon I) \geq -  \Lambda n (a + \epsilon), \quad \forall \epsilon >0.
\]
Combine the above two estimates, we obtain that for each $(x, t) \in A_{a} (E) \setminus \mathcal{N}'$,
\[
\det \smp{I +a^{-1}D^2u (x,t)} \smp{1 -a^{-1} u_{t} (x,t)} \leq  \smp{ 1 + \frac{\Lambda}{\lambda} n }^{n} \smp{1 + \Lambda n}.
\]
The desired estimate follows.
\end{proof}

\medskip

\begin{lem}
\label{Barrier}
Given $( \tilde{x}_1, \tilde{t}_1 ) \in \overline{B_{\sqrt{T_1/\theta} }(x_1)} \times \{t= t_1 + T_1\} \cap A_a (\overline{Q}_1)$, let $P_1 = P_{(y_1,s_1; a)}$ be the corresponding contact parabola. Let $Q_{r}(x_2,t_2)$ be the cylinder given in Lem.\ref{intersection1} with respect to $(x_0, t_0)$ and $(x_1, t_1) $. 

Then there exists a point $(y_0, s_0) \in Q_{r/4} (x_2, t_2)$ such that 
\[
u(y_0, s_0) \leq P_1 (y_0, s_0) + aC T_1 , \quad C \text{ universal. }
\]
\end{lem}

\begin{proof}
We shall perform a barrier argument. Recall from Lem.\ref{intersection1} that $r \geq\gamma \sqrt{T_1/\theta}$ where $\gamma = (\sqrt{2}-1 )/4$ (see Eq.\eqref{Lem22r}). Set 
\begin{equation}
\alpha = \frac{\theta}{2} \smp{4 - \frac{\gamma^2}{16} }^{-1}, \quad \delta = \frac{ \gamma^2}{ 16 } \frac{\alpha}{ \theta}.
\end{equation}
Then it is easy to check that $\wt{\PB}_2:=\PB_{T_1 (1/2+\delta)}^{\alpha} (x_2, t_2 - \delta)$ satisfies that
\[
\{ \wt{\PB}_2 \cap \{t \leq t_2 \} \}\subset \overline{Q}_{r/4}(x_2, t_2), \quad   (\tilde{x}_1, \tilde{t}_1) \in \partial \tilde{\PB}_{2} \times \{t = t_1 + T_1\}.
\]
We shall show that the desired point $(y_0,s_0)$ occurs in $ \wt{\PB}_2 \cap \{t \leq t_2\}$.

\begin{center}
\includegraphics[width=14cm, height=4.5cm]{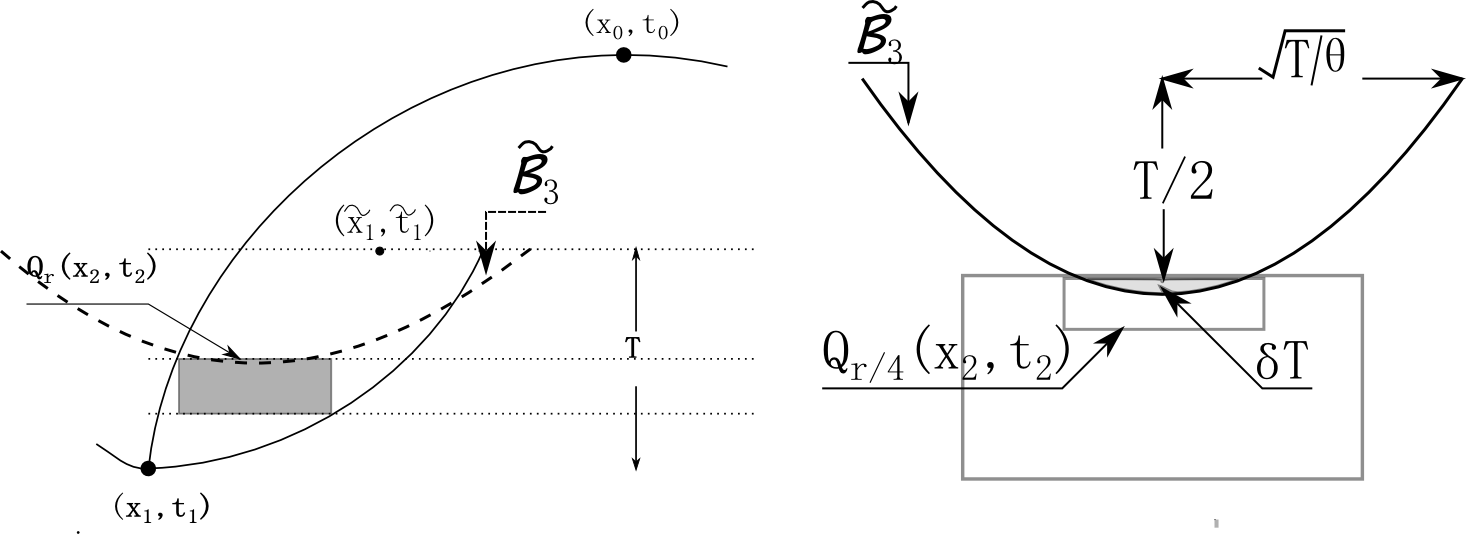}

Figure 3.1
\end{center}

Up to a translation of coordinates, we may assume that $x_2 =0, t_2 - \delta T_1=0$. Consider the following function
\[
\varphi (x, t) := a C' 
T_1 \smp{ \frac{t}{T_1} }^{-\beta_1} \smp{  e^{-\beta_2 \rho}  - e^{-\beta_2 \alpha^{-1}}} ,\quad  \rho = \frac{\abs{x}^2}{t} \leq \alpha^{-1}, \; t \geq \delta T_1,
\]
where $C', \beta_1, \beta_2$ are constants to be determined. 

By direct calculation, we have
\[
D^2 \varphi (x, t) = aC ' \smp{\frac{t}{T_1}}^{- (\beta_1 + 1)} \smp{ 4\beta_2^2 \frac{x\otimes x}{t} - 2\beta_2 I_n  } e^{- \beta_2 \rho}, 
\]
and
\[
\partial_t \varphi (x, t) = aC' \smp{\frac{t}{T_1}}^{-(1+\beta_1)} \smp{ \beta_2\rho-  ( 1-e^{\beta_2 (\rho -\alpha^{-1})} ) \beta_1}  e^{-\beta_2 \rho}.
\]

First choose $\beta_2 = \max \{\alpha^{-1}+1, \Lambda n/\lambda\}$, then choose 
\[
\beta_1 = \{\lambda \beta_2 + \Lambda n\}/ (1 - e^{-\beta_2 \alpha^{-1}/2})
\]
and finally choose 
\[
C' = (n+1)e^{\beta_2 \alpha^{-1}}.
\]
By the choice of $c_0$ and the assumption that $a \leq c_0 \delta$, we have 
\[\norm{D^2 \varphi}_{L^{\infty} (\wt{B}_3 \cap \{t \geq t_2\}) }  \leq \delta.\] 

Let $\psi  = P_1 + \varphi$. The above choice of constants (only depending on $n,\lambda,\Lambda$) ensures that 
\begin{equation}
\label{UseEq3}
F(D^2 \psi) (x, t) -\psi_t > 0,    \quad \forall (x, t) \in \wt{\PB}_2 \cap \{t \geq \delta T \}.
\end{equation}
Then by definition of a viscosity solution, the minimum of $u -\psi$ has to occur on $\partial  \wt{\PB}_2 \cap \{\delta T_1 \leq t < T_1 (1/2 + \delta)\} $. On other hand,
\[\begin{split}
& \psi (\tilde{x}_1, \tilde{t}_1) < P (\tilde{x}_1, \tilde{t}_1) = u(\tilde{x}_1,   \tilde{t}_1), \\
&  \psi (x,t) = 0, \; \forall (x,t) \in \partial  \wt{\PB}_2 \cap \{ t < T_1 (1/2 + \delta) \}.
\end{split}\]
So the minimum of $u -\psi$ has to occur on $\wt{\PB}_3 \cap \{ t = \delta T_1 \}$ and the minimum value is negative. Let $(y_0, s_0)$ be the minimum point, then
\[
u(y_0, s_0) \leq \psi (y_0, s_0) \leq P_1 (y_0, s_0) + \varphi (y_0, s_0).
\]
The desired estimate follows from the explicit expression of $\varphi$.
\end{proof}

\begin{rem}
In the above proof, $\tilde{\PB}_{2}$ may intersect $Q_1^{c}$. However, as our discussion is completely local, we may assume that $u$ satisfies Eq.\eqref{Eq1} in a larger domain, e.g. $Q_2$.
\end{rem}

Now we are ready to prove Prop.\ref{HoC}.

\begin{proof}[Proof of Prop.\ref{HoC}]
Keep the same notations as in Lem.\ref{Barrier}.  Consider now parabolas of the following form
\[
p_{y,s} (\xi, \tau ):= P_1 (\xi, \tau) + 2^{20} C\smp{ (\tau-s) - \abs{\xi-y}^2/2 },\; (y,s) \in \PB^{-1/2}_{ r^2/16} (y_0, s_0 - r^2/16).
\]

First of all, we observe that the opening of $p_{y,s}$ is $C'+1$ with $C'=2^{20}C$. Set $\delta = 1/(C'+1)$, then the vertex $(\tilde{y}, \tilde{s})$ of $p_{y,s}$ is given by
\[
\tilde{y} =\delta y_1 +(1 - \delta) y, \quad \tilde{s} = \delta s_1 + (1 - \delta) s+ h_{y_1} (y), 
\]
with 
\[
h_{y_1} (y) = \tfrac{1}{2 } \bigp{ \delta \abs{\tilde{y} - y_1}^2  + (1- \delta ) \abs{\tilde{y} - y}^2  }.
\]

Let $\tilde{E}$ be the set of vertexes of parabolas $\{p_{y,s} \; | \; (y,s ) \in\PB^{-1}_{r^2/16}  (y_0, s_0 - \tfrac{r^2}{16})  \}$. Then $\tilde{E}$ is the image of $\PB^{-1}_{r^2/16}  (y_0, s_0 - \tfrac{r^2}{16}) $ under the bijective mapping $(y,s) \mapsto (\tilde{y}, \tilde{s})$. Therefore,
\[
\abs{\tilde{E}} = \smp{ \frac{C'}{C' +1} }^{n+1} \abs{\PB^{-1}_{r^2/16}  (y_0, s_0 - \tfrac{r^2}{16} )}.
\]

Next, we claim that 
\[
A_{ (C' + 1 ) a } (\tilde{E} ) \subset Q_{r } (x_2, t_2 )\subset  \PB^{-\theta}_{T_0} (x_0, t_0)   \cap  \PB^{\theta}_{T_1} (x_1, t_1).  
\]
To show the claim, we need the following two observations:

i) Note that
\[
(s_0-s) - \tfrac{\abs{y_0-y}^2}{2}  \geq \tfrac{r^2}{16}  \geq (\sqrt{2}-1)^2 \cdot 2^{-12} T_1,\;   (y,s) \in \PB^{-1/2}_{ r^2/16} (y_0, s_0 - r^2/16).
\]
Thus, by Lem.\ref{Barrier},
\[
p_{y,s}(y_0, s_0)  \geq P_{1} (y_0 ,s_0) + C T_1 \geq u (y_0, s_0).
\]
Therefore, $p_{y,s}$ contacts $u$ before time $s_0$. 

ii) Note that
\[
(t-s) - \abs{x-y}^2/2< 0, \quad \forall  t < s_0 , (x,t ) \not \in \PB^{1/2}_{s_0 -s} (y,s).
\]
Thus
\[
p_{y,s} (x, t )  < P_1(x, t) < u(x, t),  \quad  \forall  t < s_0 , (x,t ) \not \in \PB^{1/2}_{s_0 -s} (y,s).
\]
Therefore, $p_{y,s}$ contacts $u$ in $ \PB^{1/2}_{s_0 -s} (y,s)$.

Combining the above two observations and Lem.\ref{intersection1}, we prove the claim.

\smallskip

Now we apply Lem.\ref{ABP2} to obtain that
\[
\abs{A_{(C+1) a} (\tilde{E})} \geq c_2 \abs{\tilde{E}} = c_2 \smp{ \frac{C'}{C' +1} }^{n+1}  \abs{\PB^{-1/2}_{r^2/16}  (y_0, s_0 - \tfrac{r^2}{16} )}.
\]

Finally, by the explicit formula of $(\tilde{y}, \tilde{s} ) $, it is easy to see that $\tilde{E} \subset Q_1$, hence $A_{(C+1)a }  (\tilde{E} ) \subset A_{(C+1)a}(\overline{Q}_1)$.

The desired estimate follows by choosing $c_1 $ according to $n,c_2$ and $C'$.
\end{proof}

\medskip

By a covering argument, Prop.\ref{HoC} implies the following proposition.
\begin{prop}
\label{CHoC}
Let $u \in C(\overline{B}_1 \times \R)$ be bounded and locally uniformly semi-concave and $F: \Sym(n) \rightarrow \R$ satisfy $\Hp{0}1) - \Hp{0}3)$. Let $\PB^{-\theta}_{T_0} (x_0, t_0) \subset Q_1$. Let $c_0, c_1 $ be the constants given in Prop.\ref{HoC}. 

Assume that $a \in (0, c_0 \delta)$, $(x_0, t_0) \in A_a (\overline{Q}_1 )$, $\tfrac{3}{4}\leq \theta \leq 4$ and
\[
F (D^2u ) - u_t \leq 0 \text{ in } Q_1.
\]
Then for all $k \in \N$ satisfying $c_1^{-k} a \leq c_0 \delta$, 
\[
\abs{\PB_{T_0}^{-\theta} (x_0, t_0) \setminus A_{c_1^{-k}a} (\overline{Q}_1) } \leq (1 -\eta_2 c_1)^{k} \abs{\PB_{T_0}^{-\theta} (x_0, t_0)},
\]
where $\eta_{2}$ is given in Eq.\eqref{chat}.
\end{prop}

\begin{proof}
We prove by induction on $k$. For simplicity of notations, let $A_{a} = A_{a} (\overline{Q}_1 )$. 

The case that $k=0$ is trivial. Assume that the statement is valid for $k$, we need to deduce the case for $k+1$. 

For each $(x, t) \in \PB_{T_0}^{-\theta} (x_0, t_0) \setminus A_{c_1^{-k}a}$, consider the minimal $T_{x,t}$ such that 
\[
\PB_{T_{x,t}}^{\theta} (x, t) \cap ( \PB_{T_0}^{-\theta} (x_0, t_0) \cap  A_{c_1^{-k}a}) \neq \emptyset.
\]
Since $(x_0, t_0) \in A_a$, $T_{x,t} \leq T_0$ for all $(x,t) \in  \PB_{T_0}^{-\theta} (x_0, t_0)$.

By Lem.\ref{VC}, we may extract from $\{\PB_{T_{x,t}}^{-\theta} (x,t) | \; (x, t) \in \PB_{T_0}^{-\theta} (x_0, t_0) \setminus A_{c_2^{-k}a} \}$ a sequence $\bigp{ \PB_{i} = \PB_{T_{x_i, t_i} }^{\theta} (x_i, t_i)  \; | \; i \in \N} $ such that
\[ 
\PB_{i}'s \text{ are disjoint, } \quad  \PB_{T_0}^{-\theta} (x_0, t_0) \setminus A_{c_1^{-k}a} \subset \bigcup_i \wh{\PB_i}, \quad   \abs{\PB_i }/\abs{\wh{\PB_i} }  \geq \eta_2.
\]

On other other hand, we may apply Prop.\ref{HoC} to each $\PB_i$ and obtain
\[
\abs{\PB_i} \leq  c_1^{-1} \abs{ A_{c_1^{-(k+1)} a}  \cap \PB^{-\theta}_{T_0} (x_0, t_0)   \cap  \PB_i }.
\]
Combine these two, we have
\[
\begin{split}
\abs{\PB_{T_0}^{-\theta} (x_0, t_0) \setminus A_{c_1^{-k}a}} &  \leq \sum_{i=1}^{\infty} \abs{\wh{\PB}_i} = \frac{1}{\eta_2} \sum_i^{\infty} \abs{\PB_i} 
 \leq \frac{1}{ \eta_2 c_1 } \sum_i \abs{ A_{c_1^{-(k+1)} a}  \cap \PB^{-\theta}_{T_0} (x_0, t_0)  \cap  \PB_i }  \\
& \leq \frac{1}{\eta_2 c_1} \abs{ (A_{c_1^{-(k+1)} a}  \setminus A_{c_1^{-k}a})\cap \PB^{-\theta}_{T_0} (x_0, t_0)   } 
\end{split}
\]
The desired estimate then follows immediately.
\end{proof}

\section{Decay of Oscillation}
In this section, we prove that the oscillation of $u$ decays under suitable conditions (Prop.\ref{OD}). This fact will allow us to reach higher regularity via a blow-up argument.

\begin{prop} 
\label{OD}
Let $u \in C(\overline{Q}_1)$ be a solution to Eq.\eqref{Eq1} in $Q_1$ with $F : \Sym (n) \rightarrow \R$ satisfying $\Hp{0}1)-\Hp{0}3)$. If
\[
\norm{u}_{L^{\infty} (Q_1)  } \leq c_0 \delta, 
\] 
then
\[
\osc_{Q_{1/3}} u \leq (1 - \nu_0 ) \osc_{Q_1} u,
\]
where $\nu_0$ is a universal constant.
\end{prop}

The proof is divided into several lemmas. First, we recall the sup-inf convolution (see Sec.8 in \cite{UserGuilde}).
\begin{lem}
\label{convolution}
Let $u \in C(\overline{Q}_1)$, define
\[
u_{\epsilon} (x,t)  = \inf_{(\xi,\tau) \in Q_1} \bigp{u(\xi, \tau) + \frac{1}{\epsilon} (\abs{\xi - x}^2 + (\tau - t)^2)}.
\]
Suppose that $u$ satisfies 
\[
F(D^2 u) - u_t \leq 0 \text{ in } Q_1 \text{ in viscosity sense.}
\]
Then $u_{\epsilon}$ is locally semi-concave and for every compact subset $Z \subset Q_1$, there exists $\epsilon$ depending on $(Z, u)$ such that 
\[
F_{\epsilon} (D^2 u_{\epsilon}) - u_t \leq  0 \text{ in } Z \text{ in viscosity sense,} 
\]
where
\[
F_{\epsilon} (M):= \inf \{F (M) : \abs{\xi -x} \leq \epsilon \osc_{Q_1} u , t -\epsilon \osc_{Q_1} u < \tau <t \}.
\]
Moreover, if $F$ satisfies any of $\Hp{0}1) -\Hp{0}3 )$, so does $F_{\epsilon}$.
\end{lem}

\medskip

Next we prove that the oscillation of a supersolution decays in measure.
\begin{lem}
\label{decay sup}
Let $ u \in C(\overline{Q}_1)$ and 
\[
\norm{u}_{L^{\infty} (Q_1)} \leq c_0 \delta.
\]
Suppose that $F : \Sym(n) \rightarrow \R$ satisfies $\Hp{0}1)-\Hp{0}3)$,
\[
F (D^2u ) - u_t \leq  0 \text{ in } Q_{1}.
\] 
Given $(y_0, s_0) \in Q_{1/3}$ and $\nu \in (0,1)$ such that
\[
u (y_0, s_0) \leq \nu \osc_{Q_1} u + \min_{Q_1} u.
\]
Then there exists $(x_0,t_0) \in \overline{Q}_{11/24}$ such that for all $k\in \N$ satisfying $8^2\nu c_1^{-k} \leq 1/2$,
\[
\abs{\PB^{\theta_0}_{T_0} (x_0,t_0)\cap \{u > 8^2 \nu c_1^{-k}\osc_{Q_1} u+ \min_{Q_1} u \}} \leq (1 -c_1 \eta_2)^k \abs{\PB^{\theta_0}_{T_0} (x_0,t_0)},
\]
where $T_0 = t_0 + 1$ and $\theta_0 = \min\{ \theta \; | \;  \PB_{T_0}^{\theta} (x_0, t_0 ) \subset \overline{Q}_1\}.$
\end{lem}

\begin{proof}
Extend $u$ continuously to $\overline{B}_1 \times \R$ with the same bounds. Let $\gamma = \min_{Q_1} u$ and $\delta' = \osc_{Q_1} u$. Observe that the conditions and conclusions of Lem.\ref{decay sup} are stable under uniform limit, then by virtue of Lem.\ref{convolution}, we may assume that $u $ is locally semi-concave in $Q_1$.

Let $a= 8^2 \nu \delta'$, consider 
\[
P (x,t) := - \frac{a}{2} \abs{x -y_0}^2 +a (t-(s_0 -8^{-2})).
\]

Since $\norm{u}_{L^{\infty} (Q_1)} \leq c_0 \delta $, then $
a  < c_0 \delta.$ Since $P(y_0, s_0) = \nu \delta' \geq u (y_0, s_0)-\gamma$ and $P< 0$ outside $\PB_{8^{-2}}^{1/2} (y_0, s_0-8^{-2})$, there exists a point 
\[
(x_0, t_0) \in \PB_{8^{-2}}^{1/2} (y_0, s_0-8^{-2}) \cap A_{a} (\overline{Q}_1 ) .
\]

Now consider  $\PB_0^{-} := \PB_{T_0}^{\theta_0} (x_0, t_0)$.  Since $(y_0, s_0 ) \in Q_{1/3}$, then $(x_0, t_0) \in Q_{11/24}$;  the choice of $\theta_0$ ensures that $\tfrac{3}{4} \leq \theta_0 \leq 4$. Apply Lem.\ref{CHoC}, we obtain that
\[
\abs{\PB_0 \setminus A_{c_1^{-k} a} (\overline{Q}_1)} \leq (1 -  c_1 \eta_2)^{k} \abs{\PB_0}.
\]
Lem.\ref{decay sup} then follows from the observation that if $(x,t)\in A_{c_1^{-k}a} \cap \PB_0$, then
\[
u (x,t) -\gamma  = P_{c_1^{-k}a} (x,t) \leq c_1^{-k}a .
\]
\end{proof}

\begin{proof}[Proof of Prop.\ref{OD}] Let $\gamma = \min_{Q_1} u$, $\Gamma = \max_{Q_1} u$ and $\delta' = \osc_{Q_1} u$. Let $\nu_0$ be a universal constant to be specified later.

Suppose that there exists $(y_0, s_0) \in Q_{1/3} $ such that 
\begin{equation}
\label{lp}
u(y_0, s_0) \leq \gamma + \nu_0 \delta'. 
\end{equation}
We need to rule out the existence of a point $(y_1, s_1)$ such that
\begin{equation}
\label{hp}
u (y_1 , s_1) \geq \Gamma  -\nu_0 \delta'.
\end{equation}

Argue by contradiction. Suppose that such $(y_1, s_1)$ exists. Let $k \in \N $ satisfies $8^2\nu_0 c_{1}^{-k} \leq 1/2$.

First, by Lem.\ref{decay sup}, Eq.\eqref{lp} implies that there exists $ \theta_0 \in (1,3)$ and $(x_0, t_0) \in \overline{Q}_{11/24}$ such that
\begin{equation}
\label{d1}
\abs{\PB^{\theta}_{T_0} (x_0,t_0)\cap \{u > 8^2 \nu_0 c_1^{-k}\delta' + \gamma \}} \leq (1 -c_1 \eta_2)^k \abs{\PB^{\theta}_{T_0} (x_0,t_0)}.
\end{equation}

On the other hand, apply Lem.\ref{decay sup} to $\Gamma - u$ and 
\[
G (M) : = - F(-M),
\]
we obtain that there exists $(x_1, t_1) \in \overline{Q}_{11/24}, \tfrac{3}{4} \leq \theta_1 \leq 4, T_1 = t_1 + 1$ such that
\begin{equation}
\label{d2}
\abs{\PB^{\theta_1}_{T_1} (x_1,t_1)\cap \{u <\Gamma- 8^2 \nu_0 c_1^{-k}\delta'  \}} \leq (1 -c_1 \eta_2)^k \abs{\PB^{\theta_1}_{T_1} (x_1,t_1)}.
\end{equation}

Let $\PB_0^{-} = \PB_{T_0}^{\theta_0} (x_0, t_0),  \PB_{1}^{-} :=\PB^{\theta_1}_{T_1} (x_1,t_1) $ and $
\mathcal{E} := \PB_1^{-} \cap \PB_{2}^-.$ Recall the constant $\eta_1$ given in Lem.\ref{intersection2}

First take $k$ such that
\[
(1- c_1\eta_2)^{k} < \eta_1 / 3,
\]
then choose $\nu_0$ such that
\[
8^2\nu_0 c_1^{-k} < 1/4.
\]
From Eq.\eqref{d1}, Eq.\eqref{d2} and Lem.\ref{intersection2}, we obtain that
\[
\frac{\abs{ \mathcal{E} \cap \{ u < \gamma + \delta' /4 \} } }{ \abs{\mathcal{E}} } , \frac{\abs{\mathcal{E} \cap \{ u > \Gamma - \delta' /4 \} } }{\abs{\mathcal{E}}} \geq \frac{2}{3} ,
\]
which is impossible.
\end{proof}

\medskip 

\begin{cor}
\label{scaledOD}
Let $u \in C(\overline{Q}_1)$ be a solution of \eqref{Eq1} with $F : \Sym (n) \rightarrow \R$ satisfying $\Hp{0}1)-\Hp{0}3)$. If
\[
\sqrt{  \norm{u}_{L^{\infty} (Q_1)} / (c_0 \delta) }\leq 1,
\] 
then
\[
\osc_{Q_{\rho}} u \leq 2 \rho^{\alpha_0} \osc_{Q_1} u , \quad \forall \rho \in \midp{   \sqrt{  \tfrac{\norm{u}_{L^{\infty} (Q_1)} }{ c_0 \delta } },1} ,
\]
where $\alpha_0$ is a universal constant.
\end{cor}

\begin{proof}
Apply Prop.\ref{OD} to 
\[
w (x, t) = r^{-2} u(rx, r^2 t), \quad (x, t) \in Q_1,
\]
we see that 
\[
\norm{u}_{L^{\infty}(Q_1)} \leq c_0 r^2 \delta
 \Rightarrow  \quad 
\osc_{Q_{r/3}} u \leq (1 - \nu_0) \osc_{Q_r} u.
\]
Inductively apply this result, we obtain that
\[
\norm{u}_{L^{\infty}(Q_1)} \leq c_0 3^{-2k} \delta \Rightarrow 
\quad \osc_{Q_{3^{-j}}} u \leq (1 - \nu_0)^{j} \osc_{Q_1} u,  \; \forall j \leq k+1.
\]

Let $\omega(\rho) = \osc_{Q_\rho} u$. For all $\rho \geq 3^{-(k+1)}$, there exists $j \leq k+1$ such that $3^{-(j+1)} \leq \rho < 3^{-j}.$
Thus
\[
\omega(\rho) \leq \omega (3^{-j}) \leq (1 - \nu_0)^{j} \omega(1) \leq 3^{-(j+1)\beta} (1-\nu_0)^j \rho^{\beta} \omega(1).
\]
Then by taking
\[
\alpha_0 = -\log (1 - \nu_0)/ \log 3,
\]
we obtain that
\begin{equation}
\label{rOD1}
\norm{u}_{L^{\infty} (Q_1)} \leq c_0 3^{-2k} \delta  \Rightarrow \; \osc_{Q_{\rho}} u \leq 2 \rho^{\alpha_0} \osc_{Q_1} u,  \quad \forall \rho >3^{-(k+1)}.
\end{equation}

Since $\sqrt{  \norm{u}_{L^{\infty} (Q_1)  }  /( c_0 \delta)}<1$,
there exists $k$ such that 
\[
3^{-2(k+1)} c_0 \delta  \leq \norm{u}_{L^{\infty} (Q_1) } \leq 3^{-2k} c_0\delta.
\]
The desired estimates then follows from \eqref{rOD1}.
\end{proof}

\section{Proof of Thm.\ref{Main}}
Upon obtaining Prop.\ref{OD} and Cor.\ref{scaledOD}, one can follow the proof in \cite{Savin2} line by line to deduce Thm.\ref{Main}. Here we present a slightly different argument for readers' convenience.

First, we recall the underlying idea: Let $u= \epsilon v$, then in the formal sense
\[
F(D^2 u) - u_t = \epsilon \smp{ \tr [DF(0) D^2v] - v_t} + O (\epsilon^2 \norm{D^2 v}^2),
\]
and $v$ solves the linear heat equation with constant coefficients. Therefore, $v$ and thus $u$ should be regular. This formal argument will be made rigorous via a compactness argument.

Next, we recall the following elementary fact.
\begin{lem}
\label{pwEstimate}
Let $u \in C(\overline{Q}_1)$. Suppose that there exist positive constants $(\sigma, r_0, C)$ such that for every $(x,t) \in Q_{1/2}$,  there exists a polynomial
\[
P_{x,t} (\xi,\tau)= \frac{1}{2}\xi^t M_{x,t} \xi + p_{x,t}\cdot \xi + z_{x,t} + \beta_{x,t} t  \text{ with }  \norm{M_{x,t}}, \abs{p_{x,t}}, \abs{z_{x,t}}, \abs{\beta_{x,t}} \leq C
\]
satisfying
\[
\norm{u - P_{x,t}}_{L^{\infty} ( Q_{\sigma^k r_0} ) } \leq  \sigma^{k(2+\alpha)} r_0^{2+\alpha}, \quad \forall k\in \N.
\]
Then $u \in C^{2,\alpha} (Q_{1/2} )$ and
\[
\norm{u}_{C^2 (Q_{1/2})} \leq C, \quad \norm{u}_{C^{2,\alpha}(Q_{1/2})} \leq C/(\sigma r_0)^{\alpha}.
\]
\end{lem}

\medskip 

Thm.\ref{Main} is a direct consequence of the following proposition.

\begin{prop}
\label{IQA}
For each $\alpha \in  (0,1)$, there exists a small constant $r_0$ only depending on $(\delta,K, \omega, n, \lambda,\Lambda, \alpha)$ such that the following statement holds:

For every $r< r_0$ and every solution $u$ to Eq.\eqref{Eq1} with $F  $ satisfying $\Hp{0}1)-\Hp{0}5)$, if there exists a polynomial 
\[
 P_{M,p,z, \beta}= \frac{1}{2} x^tM x + p \cdot x + z + \beta t, 
\] 
with
\[
 \norm{M}, \abs{p},\abs{z}, \abs{\beta} \leq \delta/2, \quad  F(M,p,z, 0, 0) = \beta, 
\]
such that 
\[
\norm{u - P_{M, p,z,\beta}}_{L^{\infty} (Q_r)} \leq r^{2 + \alpha},
\]
then there exists another polynomial $P_{M',p',z',\beta'}$ with 
\[
 r^2\norm{M'- M},  r\abs{p' - p},  \abs{z' - z} , r^2\abs{\beta' - \beta} \leq C  r^{2 +\alpha}, \quad F(M',p',z', 0, 0) = \beta'
\]
such that 
\[
\norm{u - P_{M',p',z',\beta'}}_{L^{\infty}(Q_{\sigma r} )} \leq (\sigma r)^{2+\alpha}, 
\]
where $C, \sigma$ are universal constants and shall be specified in the proof.
\end{prop}

We first complete the proof of Thm.\ref{Main} by assuming Prop.\ref{IQA}.

\begin{proof}[Proof of Thm.\ref{Main}] Modulo translations of coordinates, by virtue of Lem.\ref{pwEstimate}, it suffices to find $\hat{r}_0 $ and a sequence of quadratic polynomials
\[
P_k (M_k, \beta_k):=  \frac{1}{2} x^tM_k x + p_k \cdot x + z_k + \beta_k t, \quad \norm{M_k}, \abs{p_k},\abs{\beta_k} \leq \delta
\]
such that 
\[
i)\;  F(M_k,p_k, z_k, 0, 0) - \beta_k  = 0;  \quad 
ii) \;\norm{u - P^k}_{L^{\infty} (Q_{\sigma^k r_0})} \leq \sigma^{k (2 + \alpha)} \hat{r}_0^{2 + \alpha};
\] 
\[
iii)
\begin{split}
&  \sigma^{2k} \hat{r}_0^{2}\norm{M_k - M_{k+1}},  \quad \sigma^k \hat{r}_0 \abs{p_k - p_{k+1}} \leq  C \sigma^{k (2 + \alpha)} \hat{r}_0^{2+\alpha} ,\\
& \abs{z_k - z_{k+1}} , \quad \sigma^{2k} \hat{r}_0^{2}\abs{\beta_k - \beta_{k+1}} \leq C \sigma^{k (2 + \alpha)} \hat{r}_0^{2+\alpha} ,
\end{split}
\]
where $C, \sigma$ are constants determined in Eq.\eqref{rc4} and Eq.\eqref{rc5} and $\hat{r}_0 \leq C^{-2} \delta^2$.

Let $r_0$ be the constant given in Prop.\ref{IQA} and take
\begin{equation}
\label{c4}
\hat{r}_0 := \min \{r_0, C^{-2} \delta^2 \} , \quad \mu_2   =\hat{r}_0^{2 + 1/2}.
\end{equation}
We shall construct the polynomials by induction. When $k = 0$, let $P_0 = 0$. Assume that we have constructed $P_k$, then the existence of $P_{k+1}$ with the desired properties follows immediately from Prop.\ref{IQA}. This completes the proof.
\end{proof}

\medskip

Now we are left to prove Prop.\ref{IQA}. In this section, we shall consider the case that $F$ only depends on $M \in \Sym (n)$. The proof for general $F$ shall be discussed in the next section.

\begin{proof}[Proof of Prop.\ref{IQA} for special $F$] Let $w : Q_1 \rightarrow \R$ be the function such that
\[
u  (rx,r^2 t) = P_{M, p,z, \beta} (rx, r^2t) + r^{2+\alpha} w(rx, r^2t).
\]
Then $w$ satisfies the following equation
\[
\tilde{F} (D^2 w) - w_t = 0 \text{ in } Q_1,  \; \text{ where } \tilde{F} (N)  = \frac{1}{r^{\alpha}} \smp{ F(M+r^{\alpha} N) - F(M)}.
\]
Note that $\tilde{F}$ satisfies $\Hp{0}1)-\Hp{0}3)$ with $\tilde{\delta} = r^{-\alpha}\delta$ and $\norm{w}_{L^{\infty}} \leq 1$. Hence $w$ satisfies
\begin{equation}
\label{oscprop}
\osc_{Q_{\rho} (x_0, t_0)} w \leq 4 \rho^{\alpha_0}, \quad \forall   \rho \geq \sqrt{ r^{\alpha}/ c_0\delta} , (x_0, t_0 ) \in Q_1,
\end{equation}
where $\alpha_0$ is the universal constant given in Cor.\ref{scaledOD}.

Now solve the following linear problem 
\[
\begin{cases}
\tr[DF(M) D^2 h ] -h_t= 0 & \text{ in } Q_{3/4} \\
h = w & \text{ on } \partial_p Q_{3/4}, 
\end{cases}
\]
where $\partial_{p} Q_1$ is the standard parabolic boundary.  Since $w$ satisfies \eqref{oscprop} on $Q_{3/4}$ and thus on $\partial_{p} Q_{3/4}$, by linear theory, we conclude that
\begin{equation}
\label{holderControl}
\osc_{Q_{\rho}(x_0, t_0) \cap Q_{3/4}} h \leq 4 \rho^{\alpha_0} , \quad \forall   \rho \geq \sqrt{ r^{\alpha}/ c_0\delta} , (x_0, t_0) \in Q_{3/4}.
\end{equation}

Fix $\epsilon > 0$,  we may choose $r$ small enough such that
\begin{equation}
\label{rc1}
\sqrt{r^{\alpha}/c_0\delta} < \epsilon /2.
\end{equation}
Then by \eqref{holderControl}, 
\begin{equation}
\label{holderControl2}
\norm{w - h}_{\partial_{p} Q_{3/4- \epsilon}} \leq 4 \epsilon^{\alpha_0}.
\end{equation}

Meanwhile, by linear theory, in $ Q_{3/4 -\epsilon}$,
\[
\epsilon \norm{Dh},  \epsilon^2 \norm{D^2 h } ,  \epsilon^2 \norm{h_t} ,\epsilon^3 \norm{D^3 h} \leq \tilde{C} \text{ universal. }
\]
Hence,  by taking $\tilde{P} = \tilde{P}_{\tilde{M}, \tilde{p}, \tilde{z}, \tilde{\beta}}$ to be the Taylor expansion of $h$ at $(0,0)$,
\[
\norm{h  - \tilde{P}}_{L^{\infty}(Q_{\sigma})} \leq \tilde{C} \sigma^3,  \quad \forall \sigma \leq 1/2.
 \]

\smallskip 

Next we control the difference between $w$ and $h$ in $Q_{3/4-\epsilon}$ by maximum principle (equivalently, the definition of viscosity solutions). Consider
\[
h_{\mu^{\pm}} := h \pm \mu ( \abs{x}^2 - (3/4-\epsilon)^2) \mp 4 \epsilon^{\alpha_0}.
\]
By taking  $r$ small enough such that
\begin{equation}
\label{rc2}
r^{\alpha} (\tilde{C}/ \epsilon^2  + \mu)< \delta /2 ,
\end{equation}
we have 
\[
\norm{ D^2 h (x) \pm \mu I }_{L^{\infty} (Q_{3/4 -\epsilon})} \leq r^{-\alpha} \delta/2.
\]
Thus
\[
\begin{split}
\tilde{F} (D^2h_{\mu^{+}}) - \partial_t  h_{\mu^{+}} \geq  \tr [DF(M) D^2 h] - h_t   + \lambda n  \mu  -\frac{\tilde{C}}{\epsilon^2} \omega( r^{\alpha}/\epsilon^2 ) ;
\end{split}
\]
and
\[
\begin{split}
\tilde{F} (D^2h_{\mu^{-}}) - \partial_t h_{\mu^{-}} \leq   \tr [DF(M) D^2 h] - h_t   - \lambda n  \mu  + \frac{\tilde{C}}{\epsilon^2}\omega( r^{\alpha}/\epsilon^2 ) .
\end{split}
\]
By choosing $r$ small enough such that
\begin{equation}
\label{rc3}
\frac{\tilde{C}}{\epsilon^2}\omega( r^{\alpha}/\epsilon^2 )  < \lambda n \mu, 
\end{equation}
we obtain that
\[
\tilde{F} (D^2 h_{\mu^+}) - \partial_t h_{\mu^+}  \geq 0  \text{ and } \tilde{F} (D^2 h_{\mu^-}) - \partial_t h_{\mu^- } \leq 0.
\] 
Meanwhile, by \eqref{holderControl2} we have
\[
h_{\mu^{+}} \leq w  \text{ on } \partial_{p} Q_{3/4  - \epsilon} ,\quad h_{\mu^-} \geq w  \text{ on } \partial_{p} Q_{3/4 - \epsilon}.
\]
Therefore, by the maximum principle, 
\[
 \norm{w - h}_{L^{\infty}( Q_{3/4 -\delta}) } \leq \mu  + 4 \epsilon^{\alpha_3}.
\]
Then, it follows that
\[
\norm{w - \tilde{P}}_{L^{\infty} ( Q_{\sigma} )} \leq  \mu + 4 \epsilon^{\alpha_0} + \tilde{C} \sigma^3.
\]

Next, we need to give a slight modification of $\tilde{P}$ because $\tilde{F} (\tilde{M}) - \tilde{\beta}$ may not equal $0$. Introduce a new parameter $s$. Note that if $r^{-\alpha} s < \delta/2$, then
\[
\tilde{F} (\tilde{M} + s I) - \tilde{\beta} \leq  n\Lambda s + \frac{\tilde{C}}{\epsilon^2 }\omega ( r^{\alpha} (\tilde{C}/\epsilon^2 + s))
\]
and
\[
\tilde{F} (\tilde{M} + s I) - \tilde{\beta} \geq \lambda s  - \frac{\tilde{C}}{\epsilon^2 }\omega ( r^{\alpha} (\tilde{C}/\epsilon^2 + s)).
\]
Thus, by varying $s$, $\tilde{F} (\tilde{M} + s I) - \tilde{\beta}$ changes sign. Therefore, there exists $s_0$  such that
\[
\tilde{F} (\tilde{M} + s_0 I) - \tilde{\beta}=0.
\]
Moreover, the above two inequalities also show that $s_0 \rightarrow 0$ when $r \rightarrow 0$.

Now, we first choose $\sigma$ universal such that
\begin{equation}
\label{rc4}
\tilde{C}\sigma^{\alpha} < 1/4 ,
\end{equation}
 then choose $\mu, \epsilon$ universal such that
\begin{equation}
\label{rc5}
\mu \leq \sigma^{2 + \alpha} /4,  \quad  4 \epsilon^{\alpha_0} \leq \sigma^{2+\alpha} /4,
\end{equation}
then choose $r_0$ according to Eq.\eqref{rc1}, Eq.\eqref{rc2} and Eq.\eqref{rc3}. Finally, we take $r_0$ small enough so that
\[
 s_0 \leq \sigma^{\alpha}/4.
\]

In this way, we obtain that
\[
\norm{w  - (\tilde{P} + s_0 \abs{x}^2/2 )}_{L^{\infty}( Q_{\sigma} )} \leq \sigma^{2+\alpha} 
\]
with
\[
 \norm{\tilde{M} }, \abs{\tilde{p}}, \abs{\tilde{z}}, \abs{\tilde{\beta}} \leq   C \text{ universal. }
\]
The proof is completed by taking 
\[
P' (x, t) := P (x, t)   + r^{2 + \alpha}\midp{ \tilde{P} \smp{\frac{x}{r}, \frac{t}{r^2}}  + \frac{s_0}{2} \abs{\frac{x}{r}}^2}.
\]
\end{proof}

\medskip

\section{Adaption to Prove Prop.\ref{IQA} for general $F$}
Now we explain how the above proof can be modified to establish Prop.\ref{IQA} in general setting.

From the proof in \S5, we see that the only property about $w$ that we have used is the oscillation decay property (Prop.\ref{OD} and its corollary). In the case that $F$ depends on more variables, we can generalize Prop.\ref{OD} to the following form.

\begin{prop} 
\label{NOD}
Suppose that $F$ satisfies $\Hp{0}1)-\Hp{0}3)$ and $u \in C(\overline{Q}_1)$ satisfies
\[
\norm{u}_{L^{\infty} (Q_1)  } \leq c_0\delta .
\]

Then there exists a universal constant $\nu_0 \in (0,1)$ such that the following statement holds: If
\begin{equation}
\label{Hnu1}
\abs{F[u] - u_t} \leq \nu_0 c_0 \delta  \text{ in } Q_1
\end{equation}
and 
\begin{equation}
\label{Hnu2}
\norm{\nabla_p F } \leq 1, \norm{\nabla_z F}  \leq \nu_0   \text{ in } \mathcal{U}_{\delta},
\end{equation}
then
\[
\osc_{Q_{1/3}} u \leq (1 - \nu_0 ) \osc_{Q_1} u.
\]
\end{prop}

Before proving Prop.\ref{NOD}, we first explain how it leads to the proof of Prop.\ref{IQA} in the case that $F$ depends on more variables.

\begin{proof}[ Proof of Prop.\ref{IQA} for general $F$]
One simply follows the argument in \S5. Consider 
\[
\tilde{F} [w] - \partial_t w,
\]
where $w$ is given by
\[
u (rx, r^2 t) = P_{M,p,z,\beta}(rx, r^2 t) + r^{2+\alpha} w(x, t)
\]
with $(M,p,z,\beta)$ satisfying
\[
F (M,p,z,0,0) = \beta ,\quad  \norm{M}, \abs{p}, \abs{z} \leq \delta/2;
\]
and $\tilde{F}$ is given by
\[
%\label{new function}
\begin{split}
\wt{F} (N,q,v,x,t ):= \frac{1}{r^{\alpha}}\{ &  F(M +r^{\alpha} N, p + rM x +  r^{1+\alpha} q, P(rx, r^2t) +r^{2+\alpha}v, r x, r^2 t)   \\
&  - F(M, p+ rMx, P (rx, r^2t), rx, r^2 t) \}.
\end{split}
\]

Clearly $\tilde{F}$ still satisfies $\Hp{0}1)- \Hp{0}3)$; moreover, by $\Hp{0}4)$ of $F$, 
\[\begin{split}
 \norm{\nabla_{q} \wt{F}}_{L^{\infty}(\Ud)} = r \norm{\nabla_p F}_{L^{\infty} (\Ud  )} \leq r K, \quad  \norm{\nabla_{v} \wt{F}}_{L^{\infty}(\Ud)} = r^2 \norm{\nabla_z F}_{L^{\infty}(\Ud)} \leq r^{2} K  .
\end{split}\]
Thus we can take $r$ small (depending on $\nu_0 c_0 \delta$ and $K$) so that $\wt{F}$ satisfies Eq.\eqref{Hnu1} and Eq.\eqref{Hnu2}.

Now we apply Prop.\ref{NOD} to obtain the oscillation decay property of $w$. The remaining argument in \S 5 applies to the general $F$ up to trivial modifications.
\end{proof}

\smallskip

Now we explain how to modify the proof of Prop.\ref{OD} to establish Prop.\ref{NOD}. It suffices to establish the following version of local homogeneity.

\begin{prop}
\label{NHoC}
Let $c_0$ be the constant given in \eqref{Constant0} and $\PB^{-\theta}_{T_0} (x_0, t_0) \subset Q_1$. Let $F$ satisfy $\Hp{0}1) - \Hp{0}3)$ and $u \in C (\overline{B}_1 \times \R) $ be a bounded and locally uniformly semi-concave function. 

Suppose that $1 \leq \theta \leq 3$ and there exists $\nu \in (0,1)$ such that 
\[\begin{cases}
F (D^2u ) - u_t \leq \nu c_0 \delta , \quad \text{ in } Q_1, \\
\norm{\nabla_p F }_{L^{\infty}(\Ud)} \leq 1 , \quad \norm{\nabla_{z} F}_{L^{\infty}(\Ud)} \leq \nu. & 
\end{cases}
\]
Then there exists a universal constant $c_1$ such that the following statement holds:

For every $\PB_{T_1}^{\theta} (x_1, t_1)$ with $(x_1,t_1) \in \PB^{-\theta}_{T_0} (x_0, t_0)$ and every $a \in (\nu c_0, c_0 \delta) $, if
\[
(\PB_{T_1}^{\theta} (x_1, t_1) \cap \{(x, t) |\; t = t_1 + T_1\})\cap (A_a (\overline{Q}_1) \cap \PB_{T_0}^{-\theta} (x_0,t_0) ) \neq \emptyset,   
\]
then 
\[
 \abs{ A_{c_1^{-1} a} (\overline{Q}_1) \cap \PB^{-\theta}_{T_0} (x_0, t_0)   \cap  \PB^{\theta}_{T_1} (x_1, t_1)  }  \geq c_1  \abs{\PB^{\theta}_{T_1} (x_1, t_1)} .   
\]
\end{prop}

\begin{proof}
One observes that \eqref{UseEq1} and \eqref{UseEq2} are still valid under the assumption that $a \geq \nu c_0 \delta$, $\abs{p} \leq 2a$ and $\abs{z} \leq  c_0 \delta $, i.e. 
\begin{equation}
\label{NUseEq2}
\begin{split}
 2a \geq a + \nu c_0 \delta & \geq F (M - \epsilon I, p, z,x,t ) \geq F\smp{Ca e\otimes e  - (a+\epsilon)I, p, z, x, t } \\
 & \geq \lambda [ (C-1)a - \epsilon ]- \Lambda (n-1) (a + \epsilon) -  2 a - \nu c_0\delta \\
& \geq \smp{ \lambda (C-1) - \Lambda(n +3) }a - O(\epsilon).
\end{split}
\end{equation}
This is sufficient to estimate $C$ (see the proof of Lem.\ref{ABP2}) and to construct the barrier $\psi$ (see the proof of Lem.\ref{Barrier}). The rest of the proof of Prop.\ref{HoC} does not involve the usage of the equation, hence it can be directly applied to establish Prop.\ref{NHoC}.\end{proof}

Cor.\ref{cmain} follows by applying Thm.\ref{Main} to the case that $w = u - \varphi$ and 
\[
\begin{split}
G(M,p,z,x,t) & := F(D^2 \varphi + M, D \varphi + p, \varphi +z, x, t ) - F[\varphi]. \\
\end{split}
\]

\textbf{Acknowledgment:} The author would like to express his great gratitude to Prof. Ovidiu Savin, from whom the author has learnt many techniques and arguments in this paper. The author would like to express his thanks to Prof. Duong Hong Phong for his constant encouragement and helpful advice. The author also wants to thank Ye-Kai Wang who has helped to go through many technical details. The author would like to thank the referee for many helpful suggestions.

%Bibliography..............................................................................................................................
%\bibliographystyle{}
%\bibliography{}

\begin{thebibliography}{alpha}


%\bibitem{CC}
 %X. Cabr\'e, L. A. Caffarelli,
 %\emph{Fully Nonlinear Elliptic Equations}.
 %AMS Colloquium Publications Vol.43


\bibitem{Armstrong1}
S. Armstrong, L. Silvestre, C.K. Smart, 
\emph{Partial regularity of solutions of fully nonlinear uniformly elliptic equations}.
To appear in Comm. Pure Appl. Math.



\bibitem{Caffarelli1}
L.A. Caffarelli, 
\emph{A Harnack inequality approach to the regularity of free boundaries. I. Lipschitz free boundaries are $C^{1,\alpha}$.}
Rev. Mat. Iberoamericana 3 (1987), no. 2, 139-162

\bibitem{Caffarelli2}
L. A. Caffarelli,
\emph{Interior estimates for fully nonlinear equations}.
Ann. of Math. 130 (1989), 189-213


\bibitem{DS}
D. de Silva, O. Savin,
\emph{$C^{2,\alpha}$ regularity of flat free boundaries for the thin one-phase problem}.
Preprint 2011, arXiv:1111.2513.

\bibitem{EG}
L.C. Evans, R.F. Gariepy,
\emph{Measure Theory and Fine Properties of Functions}.
CRC Series, Studies in Advanced Mathematics, 1992

\bibitem{UserGuilde}
M. G. Crandall, H. Ishii, P.-L. Lions
\emph{User's guide to viscosity solutions of second order partial differential equations}.
Bull. Amer. Math. Soc. 27 (1992), 1-67



\bibitem{Krylov}
N.V. Krylov, 
\emph{Nonlinear elliptic and parabolic equations of second order}.
Mathematics and its Applications, 1985

\bibitem{Lieberman}
G. M. Lieberman,
\emph{Second order parabolic differential equations}.
World Scientific, 1996


\bibitem{Giusti}
E. Giusti,
\emph{Minimal surfaces and functions of bounded variation}.
Basel-Boston: Birkhauser Verlag, 1984


\bibitem{Savin1}
O. Savin, 
\emph{Phase transitions: regularity of flat level sets}
PhD Thesis, UT Austin, 2003

\bibitem{Savin2}
 O. Savin,
\emph{Small perturbation solutions for elliptic equations}.
 Comm. PDE 32 (2007), 557-578


\bibitem{Tso}
K. S. Tso, 
\emph{On an Aleksandrov-Bakel'man type maximum principle for second-order
parabolic equations}.
Comm. PDE 10 (1985), 543-553


\bibitem{Wang1}
L. Wang,
\emph{On the regularity theory of fully nonlinear parabolic equations. I}.
Comm. Pure Appl. Math. 45 (1992), no. 1, 27-76


\bibitem{Wang2}
L. Wang, 
\emph{On the regularity theory of fully nonlinear parabolic equations. II.}
 Comm. Pure Appl. Math. 45 (1992), no. 2, 141-178


\bibitem{Wang3}
L. Wang,
\emph{On the regularity theory of fully nonlinear parabolic equations. III.}
Comm. Pure Appl. Math. 45 (1992), no. 3, 255-262







\end{thebibliography}

% End Document ............................................................................................................................................
\end{document}